\documentclass[10pt,reqno]{amsart}
\RequirePackage[OT1]{fontenc}
\RequirePackage{amsthm,amsmath}
\RequirePackage[numbers]{natbib}
\RequirePackage[colorlinks,linkcolor=blue,citecolor=blue,urlcolor=blue]{hyperref}
\usepackage{amssymb}
\usepackage{anysize}
\usepackage{hyperref}
\usepackage{extarrows}
\usepackage{multirow}
\usepackage{natbib}
\usepackage{stmaryrd}
\usepackage{color}
\usepackage{amsmath}
\usepackage{enumitem}
\usepackage{graphicx}

\numberwithin{equation}{section}

\theoremstyle{plain} 
\newtheorem{thm}{Theorem}[section]
\newtheorem{lem}[thm]{Lemma}

\newtheorem{pro}[thm]{Proposition}

\theoremstyle{remark}
\newtheorem{rem}[thm]{Remark}

\renewcommand{\Re}{\mathrm{Re}\,}
\renewcommand{\Im}{\mathrm{Im}\,}

\newcommand{\E}{{\mathbb E }}

\newcommand{\N}{{\mathbb N}}

\newcommand{\ii}{\mathrm{i}}

\newcommand{\dd}{\mathrm{d}}

\newcommand{\bs}{\boldsymbol}

\newcommand{\nc}{\normalcolor}

\renewcommand{\mathbf}[1]{\bs{#1}}
 
\marginsize{35mm}{35mm}{38mm}{40mm}  

\begin{document}

 \begin{minipage}{0.85\textwidth}
 \vspace{2.5cm}
 \end{minipage}
\begin{center}
\large\bf
Numerical Radius of Non-Hermitian Random Matrices
\end{center}

\renewcommand{\thefootnote}{\fnsymbol{footnote}}	
\vspace{1cm}
\begin{center}
\begin{minipage}{0.3\textwidth}
\begin{center}
Zhigang Bao\footnotemark[1]  \\
\footnotesize {The University of Hong Kong}\\
{\it zgbao@hku.hk}
\end{center}
\end{minipage}
\hspace{3ex}\begin{minipage}{0.3\textwidth}
 \begin{center}
Giorgio Cipolloni\footnotemark[2]\\
\footnotesize 
{University of Rome Tor Vergata}\\
{\it cipolloni@axp.mat.uniroma2.it}
\end{center}
\end{minipage}
\footnotetext[1]{Partially supported by Hong Kong RGC Grant GRF 16304724, NSFC12222121}
\footnotetext[2]{Partially supported by the MUR Excellence Department Project MatMod@TOV awarded to the Department of Mathematics, University of Rome Tor Vergata, CUP E83C18000100006.}

\renewcommand{\thefootnote}{\fnsymbol{footnote}}	

\end{center}
\vspace{1cm}

\begin{center}
 \begin{minipage}{0.8\textwidth}\footnotesize{ {\bf Abstract}: For a square matrix, the range of its Rayleigh quotients is known as the numerical range, which is a compact and convex set by the Toeplitz-Hausdorff theorem. The largest value and the smallest boundary value (in magnitude) of this convex set are known as the numerical radius and inner numerical radius respectively. The numerical radius is often used to study the convergence rate of iterative methods for solving linear systems. In this work, we investigate these radii for complex non-Hermitian random matrix and its elliptic variants. For the former, remarkably, these radii can be represented as extrema of a  stationary Airy-like process, which undergoes a correlation-decorrelation transition from small to large time scale. Based on this transition, we obtain the precise first and second order terms of the numerical radii. In the elliptic case, we prove that the fluctuation of the numerical radii boils down to the maximum or minimum of two independent Tracy-Widom variables. }
\end{minipage}
\end{center}

 \vspace{2mm}
 
 {\small
\footnotesize{\noindent\textit{Date}: \today}
 
 \vspace{2mm}

 }

\thispagestyle{headings}

\section{Introduction}

The {\it numerical range} of a square matrix $B\in \mathbb{C}^{N\times N}$ is defined as the set of its  Rayleigh quotients
\begin{align*}
R(B)=\big\{x^*B x: x\in \mathbb{C}^N, \|x\|_2=1\big\},
\end{align*}
which is also often called the {\it field of values}. 
By the Toeplitz-Hausdorff theorem, it is known that $R(B)$ is compact and convex. 
The {\it numerical radius} of $B$ is defined as 
\begin{align*}
r_+(B)= \max\{|z|: z\in R(B)\}
\end{align*}
and the {\it inner numerical radius} of $B$ is defined as 
\begin{align*}
r_-(B)=\min \{|z|: z\in \partial R(B)\}. 
\end{align*}
Let $\rho(B)$ be the spectral radius of $B$. From the definition, one can easily see that for any square matrix $B$,
\begin{align*}
\rho(B)\leq r_+(B),
\end{align*}
as the former can be obtained by choosing $x$ in $x^*Bx$ to be the unit eigenvector associated with the largest absolute eigenvalue. Another two fundamental inequalities regarding $r_+(B)$ for any square matrix $B$ are 
\begin{align*}
\frac{1}{2}\|B\|_{\text{op}}\leq r_+(B)\leq \|B\|_{\text{op}}, \qquad 
\end{align*}
and 
\begin{align*}
r_+(B^m)\leq \big(r_+(B)\big)^m, \qquad \forall m\in \mathbb{N}. 
\end{align*}
Consequently, one has the following inequality for any $m\in \mathbb{N}$,
\begin{align*}
\|B^m\|_{\text{op}}^{1/m}\leq 2^{1/m} r_+(B). 
\end{align*} 
In contrast, the Gelfand's formula states
\begin{align*}
\lim_{m\to \infty} \|B^m\|_{\text{op}}^{1/m}=\rho(B). 
\end{align*}
The above relation holds only in the limiting sense, i.e., as $m\to \infty$. In particular, neither $\rho(B)$ nor a simple multiple of it can be used to bound $\|B^m\|_{\text{op}}^{1/m}$ for a fixed $m$. Since, in practice, any iterative algorithm can only be run for a finite number of steps, the numerical radius, in contrast to the spectral radius, is frequently employed as a more reliable indicator of the convergence rate of iterative methods for solving linear systems; see, for instance, \cite{ALP94, E93, GT82, H94, LW64, T25}. The inner numerical radius, although less studied, also has applications in generalized eigenproblems \cite{CH99, HTV02}. For a detailed introduction to the numerical range and numerical radius, we refer to the monograph \cite{GR97} and to the entire Chapter 1 of \cite{HJ94}.

Over the past few decades, the field of random matrix theory has seen tremendous progress.   A prominent model is the complex non-Hermitian random matrix $A=(a_{ij})_{1\le i,j\le N}$ with independent entries  satisfying 
\begin{align}
&\mathbb{E}a_{ij}=0, \qquad \mathbb{E}|a_{ij}|^2=\frac{1}{N}, \qquad \mathbb{E} a_{ij}^2=0,\notag\\
& \mathbb{E}|\sqrt{N}a_{ij}|^k<C_k<\infty, \quad \text{for any given } k. \label{082701}
\end{align}
A fundamental result for the matrix $A$ is the circular law, which states that the empirical distribution of its eigenvalues converges to the uniform distribution on the unit disc in the complex plane \cite{ Bai97, Girko84, PZ10, GT10, TV10b}. Regarding the local spectral statistics, the spectral radius and the operator norm of non-Hermitian random matrices have also been very well understood. 
It is now well-known that $\rho(A)$ converges to $1$ and $\|A\|_{\text{op}}$ converges to $2$ in probability, or even almost surely, under suitable moment conditions \cite{BSY88, BY93, BY86, BCG22}. In addition to the first-order limits, the finer fluctuation behavior of $\rho(A)$ and $\|A\|_{\text{op}}$ has also been thoroughly studied. In particular, it is now known from \cite{CEX23} that $\rho(A)$ follows a Gumbel law under the general assumption (\ref{082701}), a universal result that aligns with the Gaussian case \cite{Rider03}. More specifically, one has
\begin{align*}
\rho(A)=1+\sqrt{\frac{\gamma_N}{4N}}+\frac{\mathcal{G}_N}{\sqrt{4N\gamma_N}}, \qquad \gamma_N:=\log N-2\log\log N-\log 2\pi
\end{align*}
where $\mathcal{G}_N$ is asymptotically a standard Gumbel random variable.  Regarding $\|A\|_{\text{op}}$, we know that it follows the Tracy-Widom distribution from 
\cite{Joh00,Sosh02,Peche,PY} etc. More specifically, one has 
\begin{align*}
\|A\|_{\text{op}}=2+(2N)^{-2/3} \mathcal{T}_N,
\end{align*}
where $\mathcal{T}_N$ is asymptotically a standard $\text{TW}_2$ random variable. 
We also refer to the survey \cite{BC12} and the monograph \cite{BF22} for more introduction about the spectral properties of $A$. 

In contrast to the extensive study of $\rho(A)$ and $\|A\|_{\text{op}}$, research on the numerical radius of $A$ or more general non-Hermitian random matrices is much scarcer. As far as we know, the only results available in the literature are that $r_{\pm}(A)$ converges to $\sqrt{2}$ almost surely as $N\to\infty$, which can be seen as a corollary of the convergence of the entire numerical range of $A$ in \cite{CGLZ14}. We also refer to \cite{CGT25} for a related discussion. In this paper, we aim to study the finer behavior of $r_{\pm}(A)$.

A remarkable fact is that for any square matrix $B$, these radii admit the following representations, which can be found from \cite[Page 41]{HJ94} and \cite{PT02} for instance,
\begin{align}
r_+(B)= \max_{\theta\in [0,2\pi)} \lambda_{1}(H(e^{\ii \theta}B)), \qquad r_-(B)=\Big|\min_{\theta\in [0,2\pi)}\lambda_{1}(H(e^{\ii \theta}B))\Big|. \label{090901}
\end{align}
Here $H(D)=(D+D^*)/2$ is the Hermitian part of a square matrix $D$, and $\lambda_1(H(D))$ is the largest eigenvalue of an Hermitian matrix $H(D)$.  We remark here that although the numerical radii are given by (\ref{090901}), the collection $\{(\lambda_{1}(H(e^{\ii \theta}B)),\theta):\theta\in[0,2\pi)\}$ itself is not necessarily $\partial R(B)$. Actually, $R(B)$ can be obtained as the intersection of a family of half-planes
\begin{align*}
R(B)= \bigcap_{\theta\in [0,2\pi)} \mathbb{H}_{\theta}, \qquad \mathbb{H}_{\theta}:=e^{-\ii \theta} \big\{z\in \mathbb{C}: \Re z\leq \lambda_{1}(H(e^{\ii \theta}B))\big\}. 
\end{align*}

When $B$ is a random matrix, $H(e^{\ii \theta}B)$ is a Hermitian random matrix for any fixed parameter $\theta\in [0,2\pi)$. It is now well-known that under rather general assumptions on a Hermitian random matrix, its largest eigenvalue often follows the Tracy-Widom law; see \cite{AEKS20} for instance. Consequently, for a large class of random matrices $B$, $\lambda_{1}(H(e^{\ii \theta}B))$ is a stochastic process with time domain $\theta\in [0,2\pi)$ and marginal distributions given asymptotically by the Tracy-Widom law. We may call such a process an {\it Airy-like process}. The study of $r_{\pm}(B)$ thus reduces to the study of the extrema of this process over $\theta\in [0,2\pi)$. 

In particular, for the non-Hermitian matrix $A$ defined in (\ref{082701}), it is straightforward to verify that for any $\theta\in [0,2\pi)$, $H(e^{\ii \theta} A)$ is a complex Wigner matrix. The Tracy-Widom law for such matrices is well established \cite{Fo93,TW94, Sa99,TV10,EYY12}. Given the stationarity in the moments and correlation structure of $H(e^{\ii \theta} A)$, we refer to $\lambda_1(H(e^{\ii \theta} A))$ as a {\it stationary Airy-like process}. This stationarity necessitates a delicate investigation of the process's correlation structure to study its extrema.

In contrast, if $\lambda_1(H(e^{\ii \theta} B))$ is non-stationary for some other random matrix model $B$, it might happen that the first-order limit of $\lambda_1(H(e^{\ii \theta} B))$ achieves its global extrema at only a few points, making the long-range correlation structure of the fluctuation term irrelevant for estimating the extrema. In such a case, the analysis can be largely localized. As a side result, we also consider the elliptic variants of $A$, which falls into such a non-stationary category. To this end, we further consider the complex Wigner matrix $W=(w_{ij})$ with independent (up to symmetry $w_{ij}=\overline{w_{ji}}$) entries satisfying
\begin{align*}
&\mathbb{E}w_{ij}=0, \qquad \mathbb{E}|w_{ij}|^2=\frac{1}{N}, \qquad \mathbb{E} w_{ij}^2=\delta_{ij},\notag\\
& \mathbb{E}|\sqrt{N}w_{ij}|^k<C_k<\infty, \quad \text{for any given } k.
\end{align*}
The elliptic matrix considered in this paper is defined as the following random matrix
\begin{align}
A^{\gamma}=\sqrt{\gamma} W+\sqrt{1-\gamma}A,\qquad \gamma\in (0,1].  \label{090801}
\end{align}
The limit of the empirical eigenvalue distribution is called the elliptic law  and has been well studied in \cite{Girko85,Na13, NR15, AK22} etc.

\subsection{Main Results}

We have the following theorems.

\begin{thm} \label{maintheorem} Let $A$ be the non-Hermitian random matrix defined in (\ref{082701}). We have the following expansion in probability,
\begin{align}
 &r_+(A)=\sqrt{2}+\frac{(\log N)^{2/3}}{4\sqrt{2}N^{2/3}}(1+o(1)), \label{r plus} \\
 &r_-(A)=\sqrt{2}-\frac{ (\log N)^{1/3}}{2^{1/6}N^{2/3}}(1+o(1)).  \label{r minus}
\end{align}
\end{thm}

\begin{rem} It is an appealing problem to identify the next-order terms of $r_{+}(A)$ and $r_{-}(A)$, up to and including the fluctuation terms. We refer to Section \ref{s.discussion} for further discussion. In Figure \ref{fig:NRiid}, we also present a simulation result for the numerical range of $A$, assuming the entries are complex Gaussian. As shown in \cite{CGLZ14}, the entire numerical range converges to $D(0,\sqrt{2})$, the disc with radius $\sqrt{2}$.
\end{rem}

\begin{rem} In \cite{CGLZ14}, in addition to the non-Hermitian random matrix $A$, the authors also studied the numerical range of the strictly upper triangular matrix $T$ with i.i.d. complex random variables above the diagonal. Note that $H(e^{\ii \theta} T)$ is again a complex Wigner matrix, but with zero diagonal entries. Such a difference is minor for our analysis. Hence, the approach developed in this paper can also be used to compute the first two orders of the numerical radii of $T$.
\end{rem}

\begin{figure}[htbp]
    \centering
    \includegraphics[width=0.8\textwidth]{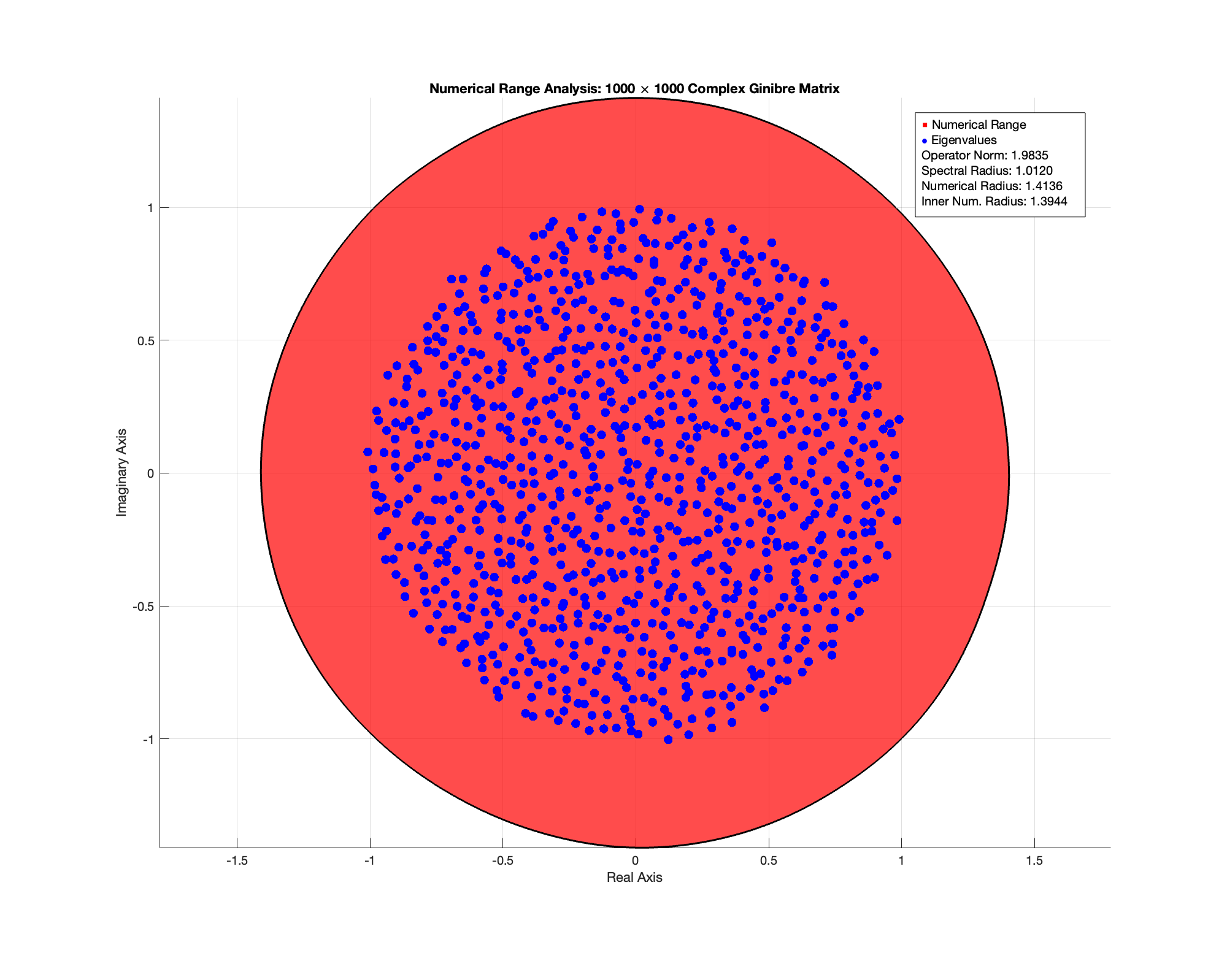}
    \caption{}
    \label{fig:NRiid}  
\end{figure}

For the elliptic model, the elliptic law in \cite{Girko85, Na13, NR15} shows that the limiting spectral distribution is supported on an ellipse. It is then natural to expect that the numerical range also has an elliptic shape. Consequently, the discussion of the  extrema of $\partial R(A^{\gamma})$ can be very much localized. Specifically, we have the following theorem regarding $r_+(A^{\gamma})$ and $r_-(A^\gamma)$, which differ at the first order.

\begin{thm}\label{thm.elliptic} Let $A^{{\gamma}}$ be the elliptic model defined in (\ref{090801}). Let $T_1, T_2$ be two independent $\text{TW}_2$ variables. 

(i) If ${\gamma}\in [N^{-1/3+\delta},1]$ for some small constant $\delta>0$, as $N\to \infty$, we have 
\begin{align*}
\frac{\sqrt{2}N^{2/3}}{\sqrt{1+{\gamma}}}\Big(r_+(A^{{\gamma}})-\sqrt{2(1+{\gamma})}\Big)\stackrel{d}\longrightarrow \max\{T_1, T_2\}. 
\end{align*}

(ii) If ${\gamma}\in [N^{-1/3+\delta},1-\delta']$ for some small constant $\delta,\delta'>0$, as $N\to \infty$, we have \begin{align*}
\frac{\sqrt{2}N^{2/3}}{\sqrt{1-{\gamma}}}\Big(r_-(A^{{\gamma}})-\sqrt{2(1-{\gamma})}\Big)\stackrel{d}\longrightarrow \min\{T_1, T_2\}. 
\end{align*}

\end{thm}

\begin{rem} It is natural to expect that the limiting behavior of $r_{\pm}(A^{{\gamma}})$ is significantly different from that of $r_{\pm}(A)$ when ${\gamma}$ is sufficiently large. Unlike the case for $A$, where $\lambda_1(H(e^{\ii\theta}A))$ is stationary in $\theta$, the elliptic model $A^{{\gamma}}$ has $\lambda_1(H(e^{\ii\theta} A^{{\gamma}}))$ reaching its maximum near $\theta=0$ or $\pi$ and its minimum near $\theta=\frac{\pi}{2}$ or $\frac{3\pi}{2}$. Thus, these extrema are much more localized. When $t<N^{-1/3-\delta}$, we can no longer localize the discussion, and we expect a phase transition to arise when $t$ crosses the scale $N^{-1/3}$; see Section \ref{s.discussion} for further details. For $r_-(A^{\gamma})$, we restrict the discussion to $\gamma<1-\delta'$, because as $\gamma\to 1$, the matrix $A^\gamma$ approaches a Wigner matrix and $r_-(A^{\gamma})$ degenerates. We also refer to Figure \ref{fig:NRElliptic} for a simulation of the numerical range for the elliptic model with Gaussian entries and ${\gamma}=0.4$. Both the support of the empirical spectral distribution and the numerical range have elliptic shapes. 
\end{rem}

\begin{figure}[htbp]
    \centering
    \includegraphics[width=0.8\textwidth]{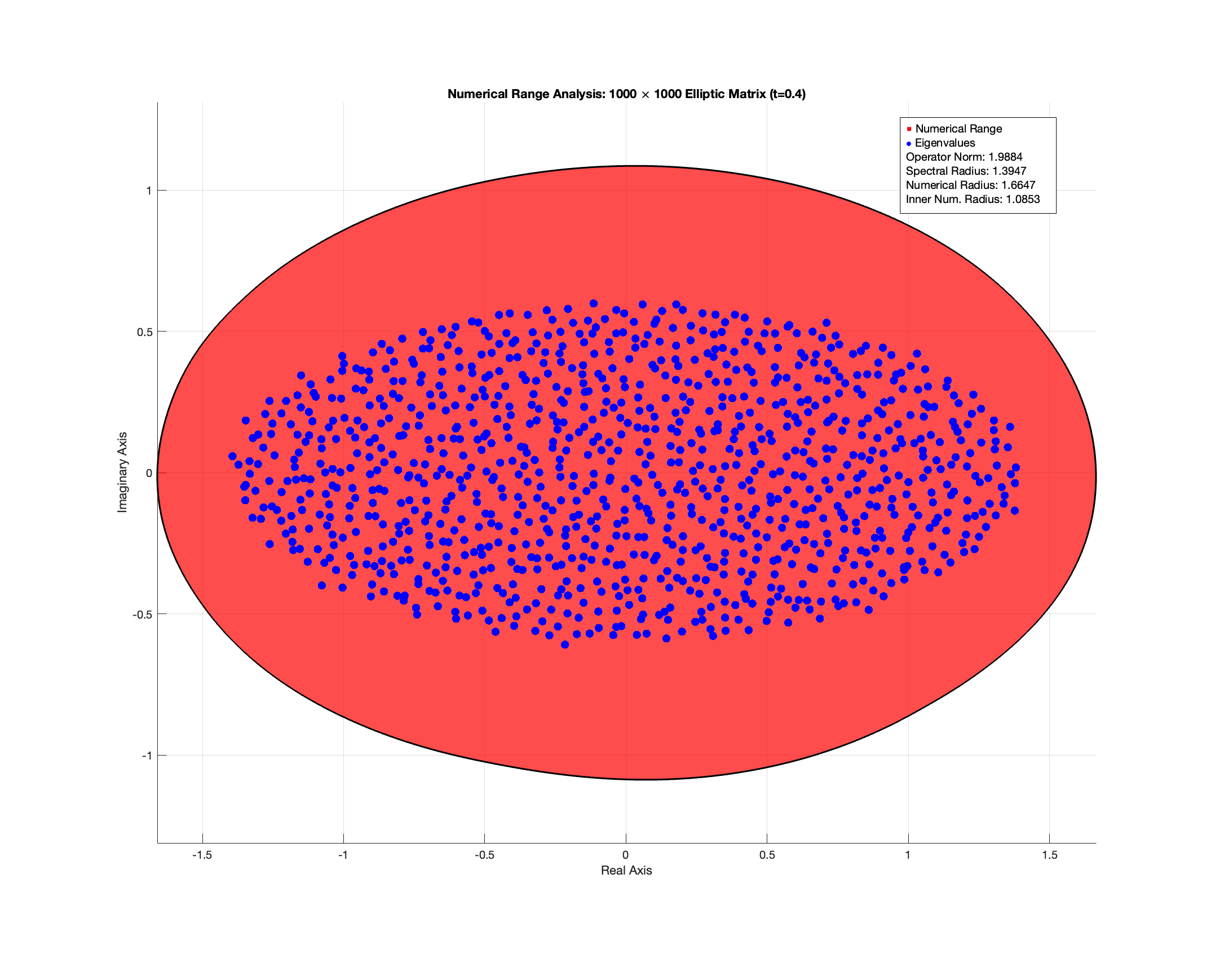}
    \caption{}
    \label{fig:NRElliptic}  
\end{figure}

\subsection{Proof  Strategy}   In this section, we briefly describe our proof strategy and further introduce some related references. 
We start with the basic fact 
\begin{align}
H(\theta)\equiv (h_{ij}(\theta)):=H(e^{\ii\theta} A)=\cos\theta \frac{A+A^*}{2}+\sin\theta \frac{ \ii (A-A^*)}{2}. \label{091102}
\end{align}
It is also elementary to check 
\begin{align}
& \mathbb{E}h_{ij}(\theta)=0, \qquad \mathbb{E}\big|h_{ij}(\theta)\big|^2=\frac{1}{2N}, \qquad \mathbb{E}\big(h_{ij}(\theta)\big)^2=\frac{1}{2N}\delta_{ij},\notag\\
& \mathbb{E}h_{ij}(\theta)\overline{h_{ij}(\theta')}=\frac{\cos(\theta-\theta')}{2N},\qquad \mathbb{E}h_{ij}(\theta)h_{ij}(\theta')=\frac{\cos(\theta-\theta')}{2N}\delta_{ij}.  \label{091101}
\end{align}
When $a_{ij}$'s are standard complex Gaussian, i.e., $A$ is a complex Ginibre matrix, $H(\theta)$ is a stationary matrix-valued Gaussian process with correlation given by (\ref{091101}). Notice that in the complex Gaussian case, $(A+A^*)/\sqrt{2}$ and $\ii (A-A^*)/\sqrt{2}$ are two independent GUE. The independence is generally not true for more generally distributed $A$, as in (\ref{090901}).

To study the extrema of $\lambda_1(H(\theta))$, the key step is to understand the correlation structure of this process. For the two-point correlation structure, we can embed the process $H(\theta)$ into a matrix-valued Ornstein-Uhlenbeck (OU) process. Consequently, the two-point correlation function of the eigenvalue process $\lambda_1(H(\theta))$ reduces to that of the largest eigenvalue of the Dyson Brownian motion (DBM) after an elementary transformation of the time domain. From the analysis of DBM, we consequently obtain a {\it correlation-decorrelation transition} for the process $\lambda_1(H(\theta))$. Roughly speaking, we can show that the fluctuations of $\lambda_1(H(\theta))$ and $\lambda_1(H(\theta'))$ are almost synchronized when $|\theta-\theta'|\ll N^{-1/6}$, while they become almost independent when $|\theta-\theta'|\gg N^{-1/6}$. Heuristically, the correlation estimate allows one to group the values of $\lambda_1(H(\theta))$ together as a single random variable for all $\theta$ in an interval of length smaller than $N^{-1/6}$. The problem of finding the extrema of the process is thus discretized to finding the extrema of a set of random variables, where the size of the set is slightly larger than $N^{1/6}$. The decorrelation estimate then allows us to treat this set as a family of random variables with short-range dependence. After a further sparsification of this set, we can obtain a subset of almost independent random variables. The extremal behavior of this subset turns out to determine the extrema of the entire set, and thus of the whole process, up to the second-order term. For generally distributed $A$, we perform a Green function comparison (GFT) to extend the result from the Gaussian case. 

It is worth mentioning that the correlation-decorrelation transition must be investigated in the small deviation regime of $\lambda_1(H(\theta))$, rather than at the typical fluctuation scale. This is because the extrema of a large number of (almost) independent random variables depends on the tail probability within a certain small deviation regime. For our problem, we need to consider the right-tail events $\{\lambda_1(H(\theta))-\sqrt{2}\geq x(\log N)^{2/3}/N^{2/3}\}$ for $r_+(A)$, and the left-tail events $\{\lambda_1(H(\theta))-\sqrt{2}\leq -x (\log N)^{1/3}/N^{2/3}\}$ for $r_{-}(A)$, where $x\in [K^{-1} ,K]$ is any positive constant. Probability estimates for such tail events have been recently derived in \cite{BCEHK25b}, based on results for the Gaussian ensemble in \cite{PZ17, BBBK} and the comparison approach developed in \cite{EX23}.

We remark that the approach used in this paper resembles the very recent work on the Law of the Fractional Logarithm (LFL) for Wigner matrices \cite{BCEHK25b} and also the previous result for the GUE case \cite{PZ17,BBBK2}. For the LFL, one considers the Wigner minor process, a sequence of Wigner matrices with increasing dimensions, which can be viewed as the growing upper-left corners of a doubly infinite array of independent (up to symmetry) random variables. In \cite{BCEHK25b,PZ17,BBBK2}, the running parameter is the dimension $N$. The key input for proving the LFL is a correlation-decorrelation transition of the largest eigenvalue of Wigner matrices of varying dimensions. More specifically, the transition occurs when $k \sim N^{2/3}$, where the two dimensions are $N$ and $N-k$, respectively. Again, the transition must be analyzed in the same small deviation regime as in our current work. Here, for $r_{\pm}(A)$, the running parameter is $\theta$, and the evolution of the matrix process is certainly different from the Wigner minor process. In general, the correlation-decorrelation transition is expected to hold universally when a random matrix is deformed to a certain extent, regardless of the specific form of the deformation.  We point out that while the general structure of the proof is similar to \cite{BCEHK25b}, there are substantial differences. In fact, we mostly focus on the analysis in the Gaussian case when we capitalize on an embedding of our process \eqref{091102} in the Dyson Brownian motion flow. Hence, in the Gaussian case our proof is rather elementary and is achieved without relying on heavy machineries nor on the integrability of the model.  In particular, no control on eigenvector overlaps as in \cite[Proposition 3.3]{BCEHK25b} is needed. We then use a GFT argument to conclude.

Finally, for the elliptic model, we can localize the analysis to two vicinities around $0$ and $\pi$ for $r_{+}(A^{\gamma})$, and to those around $\frac{\pi}{2}$ and $\frac{3\pi}{2}$ for $r_{-}(A^{\gamma})$. The correlation estimates, together with the asymptotic independence between the largest and smallest eigenvalues of a Wigner matrix, then directly lead to our results.

\subsection{Notations}

For positive quantities $f,g$ we write $f\lesssim g$ or $g\gtrsim f$ (or $f=O(g)$) and $f\sim g$ if $f \le C g$ or $c g\le f\le Cg$, respectively, for some constants $c,C>0$ which may depend on the $C_k$'s from \eqref{082701}.  We also use  $\|B\|_{\text{op}}$ to represent the operator norm of a square matrix $B$. 

We will use the concept of ``with very high probability'' \emph{(w.v.h.p.)} meaning that for any fixed $C>0$, the probability of an $N$-dependent event is bigger than $1-N^{-C}$ for $N\ge N_0(C)$. Moreover, even if not explicitly stated, every estimate involving $N$ is understood to hold for $N$ being sufficiently large. \nc  We also introduce the notion of \emph{stochastic domination} (see e.g.~\cite{EKYY13}): given two families of non-negative random variables
\[
X=\left(X^{(N)}(u) : N\in\N, u\in U^{(N)} \right) \quad \mathrm{and}\quad Y=\left(Y^{(N)}(u) : N\in\N, u\in U^{(N)} \right)
\] 
indexed by $N$ (and possibly some parameter $u$  in some parameter space $U^{(N)}$), 
we say that $X$ is stochastically dominated by $Y$, if for all $\xi, C>0$ we have 
\begin{equation}
	\label{stochdom}
	\sup_{u\in U^{(N)}} \mathbb{P}\left[X^{(N)}(u)>N^\xi  Y^{(N)}(u)\right]\le N^{-C}
\end{equation}
for large enough $N\ge N_0(\xi,C)$. In this case we use the notation $X\prec Y$ or $X= \mathcal{O}_\prec(Y)$.

\subsection{Outline} The rest of the paper will be organized as follows: In Section \ref{s.proof of main}, we prove our main result, Theorem \ref{maintheorem}, assuming the correlation-decorrelation transition. Sections \ref{s.correlation} and \ref{s.decorrelation} are then devoted to the proof of the correlation estimates and the decorrelation estimates in two regimes, respectively. In Section \ref{s.elliptic}, we prove the result for the elliptic model, i.e, Theorem \ref{thm.elliptic}. Finally, in Section \ref{s.discussion}, we discuss some further directions of study.

\section{Proof of Theorem~\ref{maintheorem}} \label{s.proof of main}
In this section, we prove Theorem \ref{maintheorem}, based on Propositions \ref{lem.correlation} and \ref{lem.decorrelation}, whose proofs  will be deferred. 

As the discussion is identical across the four quadrants, we will focus on the first quadrant $\theta\in[0, \frac{\pi}{2}]$ in the sequel. 
First, let $\varepsilon>0$ be an arbitrarily small constant, and denote
\begin{align*}
\delta_N^{\varepsilon}=N^{-\frac{1}{6}-\varepsilon},\qquad L_N^{\varepsilon}:=\big\lceil \frac{\pi}{2} N^{\frac{1}{6}-\varepsilon} \big\rceil,\qquad  K_N^{\varepsilon}:=\big\lceil \frac{\pi}{2} N^{\frac{1}{6}+\varepsilon}\big\rceil. 
\end{align*}
 We decompose $[0,\frac{\pi}{2}]$ into 
\begin{align*}
[0,\frac{\pi}{2}]=\bigcup_{i=1}^{K_N^{\varepsilon}} I_i, \qquad I_i=[(i-1)\delta_N^{\varepsilon},i\delta_N^{\varepsilon}], \qquad \forall 1\leq i\leq K_N^{\varepsilon}-1,
\end{align*}
and set $I_{K_N^{\varepsilon}}=[(K_N^{\varepsilon}-1)\delta_N^{\varepsilon}, \frac{\pi}{2}]$. 
We then further regroup these $I_i$'s into $N^{2\varepsilon}$ groups as 
\begin{align*}
\mathcal{I}_j=\bigcup_{a=0}^{L_N^{\varepsilon}} I_{j+aN^{2\varepsilon}},\qquad \forall 1\leq j\leq N^{2\varepsilon}, 
\end{align*}
where $I_i=\emptyset$ if $i>K_N^{\varepsilon}$.

For brevity, we further adopt the notation
\begin{align*}
\mathcal{H}(\theta)=\sqrt{2}H(\theta), \qquad   \lambda_1(\theta)=\sqrt{2}\lambda_{1}(H(\theta))=\lambda_1(\mathcal{H}(\theta)).
\end{align*}
We scaled the matrix (and thus its eigenvalues) by $\sqrt{2}$ to match the usual GUE scaling. This allows us to refer to existing results more easily.

 We shall first prove that for each $j$, and for any $\epsilon_1, \epsilon_2>0$, by choosing $\varepsilon$ sufficiently small
\begin{align}
\mathbb{P}\Big(\Big|\max_{\theta\in \mathcal{I}_j} \lambda_1(\theta)-2-\frac{(\log N)^{2/3}}{4N^{2/3}}\Big|\geq \epsilon_1 \frac{(\log N)^{2/3}}{N^{2/3}}\Big)\leq N^{-3\varepsilon},  \label{081101}
\end{align}
and 
\begin{align}
\mathbb{P}\Big(\Big|\min_{\theta\in \mathcal{I}_j} \lambda_1(\theta)-2+\frac{ 2^{1/3}(\log N)^{1/3}}{N^{2/3}}\Big|\geq \epsilon_2\frac{(\log N)^{1/3}}{N^{2/3}} \Big)\leq N^{-3\varepsilon}  \label{082810}
\end{align}
hold for sufficiently large $N$. 
We can then take a union over all $N^{2\varepsilon}$ indices $j$'s and conclude the proof.  

 The main strategy is to prove a correlation-decorrelation transition for the pair $(\lambda_1(\theta),\lambda_1(\theta'))$ when $|\theta-\theta'|$ goes across $N^{-1/6}$. More specifically, we have the following two lemmas. 

\begin{pro}[Correlation Regime]\label{lem.correlation} Denote by $\theta_i$ the middle point of $I_i$. For any constant $D>0$, we have $\tilde{\varepsilon}\equiv \tilde{\varepsilon}(\varepsilon, D)>0$, such that for any sufficiently large $N\geq N_0(D,\varepsilon)$, 
\begin{align}
&\mathbb{P} \Big(\Big|\max_{\theta\in I_i}\lambda_1(\theta)- \lambda_1(\theta_i)\Big|\geq N^{-\frac23-\tilde{\varepsilon}}\Big)\leq N^{-D},\notag\\
& \mathbb{P} \Big(\Big|\min_{\theta\in I_i}\lambda_1(\theta)- \lambda_1(\theta_i)\Big|\geq N^{-\frac23-\tilde{\varepsilon}}\Big)\leq N^{-D}.
\label{correlation}
\end{align}
The same conclusion is still true if we replace $\theta_i$ by any other point $\theta\in I_i$. 
\end{pro}

\begin{rem} The above lemma states that within each $I_i$, all the $\lambda_1(\theta)$ can be viewed as the same, up to a negligible error. Hence, all the $\lambda_1(\theta)$ are strongly correlated within an $I_i$. 
\end{rem}

\begin{pro}[Decorrelation Regime]\label{lem.decorrelation}  Let $\theta, \theta'\in [0,2\pi)$ such that 
\begin{align*}
N^{-\frac{1}{6}+\varepsilon}\leq \theta'-\theta \leq \frac{\pi}{2}. 
\end{align*}
Define the tail events
\begin{align*}
&\mathcal{F}_{\theta}(x)=\Big\{N^{2/3}(\lambda_1(\theta)-2)\geq x (\log N)^{2/3}\Big\}, \notag\\
& \mathcal{E}_{\theta}(x)=\Big\{N^{2/3}(\lambda_1(\theta)-2)\leq -x (\log N)^{1/3}\Big\}.
\end{align*}
Fix a large constant $K>0$. Then there exists a constant $\delta>0$,  such that 
\begin{align*}
\mathbb{P}\Big(\mathcal{F}_{\theta}(x)\cap \mathcal{F}_{\theta'}(x')\Big)=\mathbb{P}\big(\mathcal{F}_{\theta}(x)\big)\mathbb{P}\big(\mathcal{F}_{\theta'}(x')\big)\big(1+O(N^{-\delta})\big)
\end{align*}
and 
\begin{align*}
\mathbb{P}\Big(\mathcal{E}_{\theta}(x)\cap \mathcal{E}_{\theta'}(x')\Big)=\mathbb{P}\big(\mathcal{E}_{\theta}(x)\big)\mathbb{P}\big(\mathcal{E}_{\theta'}(x')\big)\big(1+O(N^{-\delta})\big),
\end{align*}
uniformly for $K^{-1}\leq x, x'\leq K$, $\ell=1,2$. 
Further, when $\theta'-\theta\geq N^{-\frac{1}{12}}$ (say), $\delta$ is independent of $\varepsilon$. 
\end{pro}

  We will further need  the following tail probability in the small deviation regime from \cite{BCEHK25b}. 
\begin{pro}\label{pro.tail} Fix any $K>0$. We have the following small deviation tail estimates.

\noindent(i)[Right tail] For any $\tau>0$ there exists a constant $C=C(K,\tau)$ such that for any $1\leq x\leq K(\log N)^{2/3}$, we have 
\begin{align}
C^{-1}\exp\Big(-\frac{4}{3}(1+\tau)x^{3/2}\Big)\leq \mathbb{P}\Big(N^{2/3}(\lambda_1(\theta)-2)\geq x\Big)\leq C\exp\Big(-\frac{4}{3}(1-\tau)x^{3/2}\Big)  \label{082811}
\end{align}
for any sufficiently large $N\geq N_0(K,\tau)$. 

\noindent(ii)[Left tail] For any $\tau>0$ there exists a constant $C=C(K,\tau)$ such that for any $1\leq x\leq K(\log N)^{1/3}$, we have 
\begin{align*}
C^{-1}\exp\Big(-\frac{1}{12}(1+\tau)x^{3}\Big)\leq \mathbb{P}\Big(N^{2/3}(\lambda_1(\theta)-2)\leq -x\Big)\leq C\exp\Big(-\frac{1}{12}(1-\tau)x^{3}\Big) 
\end{align*}
for any sufficiently large $N\geq N_0(K,\tau)$. 
\end{pro}

\begin{rem} The right tail estimate was  first established for GUE in \cite{PZ17}. Then, the upper bound on both tails were extended to general Wigner matrices in \cite{EX23}. Further, the above estimates were proved in \cite{BBBK} for both tails for Gaussian ensembles.  The extension of the results in \cite{BBBK} to general Wigner matrices was done in \cite{BCEHK25b}, using the approach in \cite{EX23}. 
\end{rem}

With the above three propositions, we can now state the proof of Theorem~\ref{maintheorem} below. 

\begin{proof}[Proof of Theorem~\ref{maintheorem}] 

We first prove (\ref{081101}) and (\ref{082810}).  As their  proofs  are similar, we will state the details for (\ref{081101}) only in the sequel. As a consequence of (\ref{082811}), we have for any $c>0$, 
\begin{align*}
C^{-1}N^{-\frac{4}{3}(1+\tau)c^{3/2}}\leq \mathbb{P}\Big(N^{2/3}(\lambda_1(\theta)-2)\geq c(\log N)^{2/3}\Big)\leq CN^{-\frac{4}{3}(1-\tau)c^{3/2}}. 
\end{align*}
Recall that $\theta_i$ is the middle point of $I_i$. Then for each $j=1, \ldots, N^{2\varepsilon}$, we have 
\begin{align*}
&\mathbb{P}\Big(N^{2/3}(\max_{i:I_i\subset \mathcal{I}_j}\lambda_1(\theta_i)-2)\geq c(\log N)^{2/3}\Big)\notag\\
&\leq L_N^{\varepsilon}\mathbb{P}\Big(N^{2/3}(\lambda_1(\theta)-2)\geq c(\log N)^{2/3}\Big)\notag\\
&\leq C L_N^{\varepsilon}N^{-\frac{4}{3}(1-\tau)c^{3/2}}
\end{align*}
For any $c_1>\frac{1}{4}$, by choosing $\varepsilon$ and $\tau$ sufficiently small, we have 
\begin{align}
\mathbb{P}\Big(N^{2/3}(\max_{i:I_i\subset \mathcal{I}_j}\lambda_1(\theta_i)-2)\geq c_1(\log N)^{2/3}\Big)\leq N^{-5\varepsilon}. \label{082910}
\end{align}
On the other hand, for any $c_2<\frac{1}{4}$, we shall show 
\begin{align}
\mathbb{P}\Big(N^{2/3}(\max_{i:I_i\subset \mathcal{I}_j}\lambda_1(\theta_i)-2)\geq c_2(\log N)^{2/3}\Big)\geq 1-N^{-5\varepsilon}.\label{081203}
\end{align} 
To this end, we recall the notation $\mathcal{F}_{\theta}(x)$ from Proposition \ref{lem.decorrelation}, and adopt the abbrevtaion
\begin{align*}
\mathcal{F}_i\equiv \mathcal{F}_{\theta_i}(c_2)
\end{align*}
in the sequel. 
We only need to estimate 
\begin{align*}
\mathbb{P}\Big(\sum_{i:I_i\subset \mathcal{I}_j}\mathbf{1}(\mathcal{F}_i)\geq 1 \Big)
\end{align*}
Note that for any fixed $c_2<\frac{1}{4}$, choosing $\varepsilon$ and $\tau$ sufficiently small, we have 
\begin{align}
\mathbb{E} \sum_{i:I_i\subset \mathcal{I}_j}\mathbf{1}(\mathcal{F}_i)=\sum_{i:I_i\subset \mathcal{I}_j}\mathbb{P}(\mathcal{F}_i)= L_N^{\varepsilon} \mathbb{P}(\mathcal{F}_1)\gtrsim N^{\frac{1}{6}-\varepsilon-\frac{4}{3}(1+\tau)c_2^{3/2}}\to \infty,  \text{  as  } N\to \infty. \label{081202}
\end{align}
Hence, by Chebyshev's inequality,
\begin{align}
\mathbb{P}\Big(\sum_{i:I_i\subset \mathcal{I}_j}\mathbf{1}(\mathcal{F}_i)=0 \Big) &\leq \mathbb{P} \Big(\Big|\sum_{i:I_i\subset \mathcal{I}_j}\big(\mathbf{1}(\mathcal{F}_i)-\mathbb{P}(\mathcal{F}_i)\big)\Big|>\frac{ \sum_{i:I_i\subset \mathcal{I}_j}\mathbb{P}(\mathcal{F}_i)}{2}\Big)\notag\\
&\lesssim \Big(\sum_{i:I_i\subset \mathcal{I}_j}\mathbb{P}(\mathcal{F}_i)\Big)^{-2}\text{Var}\Big( \sum_{i:I_i\subset \mathcal{I}_j}\mathbf{1}(\mathcal{F}_i) \Big). \label{081201}
\end{align}
Notice that 
\begin{align}
\text{Var}\Big( \sum_{i:I_i\subset \mathcal{I}_j}\mathbf{1}(\mathcal{F}_i) \Big)=\sum_{i:I_i\subset \mathcal{I}_j}\text{Var}(\mathbf{1}(\mathcal{F}_i))+\sum_{k\neq \ell}\Big(\mathbb{P}(\mathcal{F}_k\cap \mathcal{F}_{\ell})-\mathbb{P}(\mathcal{F}_k)\mathbb{P}(\mathcal{F}_{\ell})\Big), \label{082905}
\end{align}
where the sum over all $k\neq \ell$ is subject to $I_k,I_\ell\subset\mathcal{I}_j$. We nevertheless omit this constraint from the notation hereafter for brevity. 
First, we have
\begin{align}
\sum_{i:I_i\subset \mathcal{I}_j}\text{Var}(\mathbf{1}(\mathcal{F}_i))\leq \sum_{i:I_i\subset \mathcal{I}_j}\mathbb{P}(\mathcal{F}_i). \label{082902}
\end{align}
Further, we perform the decomposition into the short range sum and the long range sum 
\begin{align}
\sum_{k\neq \ell}\Big(\mathbb{P}(\mathcal{F}_k\cap \mathcal{F}_{\ell})-\mathbb{P}(\mathcal{F}_k)\mathbb{P}(\mathcal{F}_{\ell})\Big)=&\sum_{k}\sum_{\ell: 1\leq |\ell-k|\leq N^{\frac{1}{12}}}\Big(\mathbb{P}(\mathcal{F}_k\cap \mathcal{F}_{\ell})-\mathbb{P}(\mathcal{F}_k)\mathbb{P}(\mathcal{F}_{\ell})\Big)\notag\\
&+\sum_{k}\sum_{\ell: |\ell-k|> N^{\frac{1}{12}}}\big(\mathbb{P}(\mathcal{F}_k\cap \mathcal{F}_{\ell})-\mathbb{P}(\mathcal{F}_k)\mathbb{P}(\mathcal{F}_{\ell})\big). \label{082901}
\end{align}
For the first term in the RHS of (\ref{082901}), by the decorrelation estimate in Proposition \ref{lem.decorrelation}, we can bound it crudely by 
\begin{align}
\sum_{k}\sum_{\ell: 1\leq |\ell-k|\leq N^{\frac{1}{12}}}\Big(\mathbb{P}(\mathcal{F}_k\cap \mathcal{F}_{\ell})-\mathbb{P}(\mathcal{F}_k)\mathbb{P}(\mathcal{F}_{\ell})\Big)\lesssim L_N^{\varepsilon}N^{\frac{1}{12}}\Big(\mathbb{P}(\mathcal{F}_1)\Big)^2\sim N^{\frac{1}{6}-\varepsilon+\frac{1}{12}-\frac{8}{3}(1-\tau)c_2^{3/2}}. \label{082903}
\end{align}
For the second term in the RHS of (\ref{082901}), again, by decorrelation estimate in Proposition \ref{lem.decorrelation}, we can bound it by 
\begin{align}
&\sum_{k}\sum_{\ell: |\ell-k|> N^{\frac{1}{12}}}\big(\mathbb{P}(\mathcal{F}_k\cap \mathcal{F}_{\ell})-\mathbb{P}(\mathcal{F}_k)\mathbb{P}(\mathcal{F}_{\ell})\big)\notag\\
&\lesssim N^{-\delta} \sum_{k}\sum_{\ell: |\ell-k|> N^{\frac{1}{12}}} \mathbb{P}(\mathcal{F}_k)\mathbb{P}(\mathcal{F}_{\ell})\lesssim N^{-\delta}\big(L_N^{\varepsilon} \mathbb{P}(\mathcal{F}_1)\big)^2, \label{082904}
\end{align}
where $\delta>0$ is independent of $\varepsilon$. 

Plugging (\ref{082902})-(\ref{082904}) to bound the variance in (\ref{082905}), and further applying this bound together with the expectation estimate in (\ref{081202}) to (\ref{081201}), we can conclude that 
\begin{align*}
\mathbb{P}\Big(\sum_{i:I_i\subset \mathcal{I}_j}\mathbf{1}(\mathcal{F}_i)=0 \Big)\lesssim N^{-\delta}+N^{-\frac{1}{12}+\varepsilon+\frac{16}{3}\tau c_2^{3/2}}.
\end{align*}
Since $\delta$ is independent of $\varepsilon$, we can always choose $\varepsilon$ and $\tau$ sufficiently small so that (\ref{081203}) is true. 

With (\ref{082910}) and (\ref{081203}), we can conclude that for any $\epsilon_1>0$, by choosing $\varepsilon$ sufficiently small, 
\begin{align*}
\mathbb{P}\Big(\Big|\max_{i: I_i\subset \mathcal{I}_j} \lambda_1(\theta_i)-2-\frac{(\log N)^{2/3}}{4 N^{2/3}}\Big|\geq \epsilon_1 \frac{(\log N)^{2/3}}{N^{2/3}}\Big)\leq N^{-4\varepsilon}. 
\end{align*}
Further, with the correlation estimate in Proposition \ref{lem.correlation}, we can replace all $\lambda_1(\theta_i)$ by $\max_{\theta\in I_i} \lambda_1(\theta)$ in the above estimate and conclude the proof of (\ref{081101}) by slightly adjusting the value of $\epsilon_1$. Finally, by considering a union of the event in (\ref{081101}) over all $\mathcal{I}_j$'s and a further extension to the union of similar events over all 4 quadrants, we can conclude the proof of (\ref{r plus}). The proof of (\ref{r minus}) is similar. The only difference is that instead of the right tail estimate, we need to use the left tail estimate in Proposition \ref{pro.tail} (ii).  Hence, we omit the details. This conclude the proof of Theorem \ref{maintheorem}. 

\end{proof}

\section{Correlation Regime: Proof of Proposition \ref{lem.correlation}} \label{s.correlation}
In this section, we establish the desired estimate in the correlation regime, i.e., (\ref{correlation}). We  claim that for any $\theta, \tilde{\theta}$ such that $|\theta-\tilde{\theta}|\leq N^{-\frac16-\varepsilon}$, we have
\begin{align}
\mathbb{P}\Big(\Big|\lambda_1(\theta)-\lambda_1(\tilde{\theta})\Big|\geq N^{-2/3-\tilde{\varepsilon}}\Big)\leq N^{-D}\label{correlation-two-theta},
\end{align}
for some $\tilde{\varepsilon}=\tilde{\varepsilon}(\varepsilon, D)>0$.
 The proof of the claim \eqref{correlation-two-theta} is presented after concluding the proof of  \eqref{correlation}. 

\begin{proof}[Proof of \eqref{correlation}]

Given \eqref{correlation-two-theta}, to conclude \eqref{correlation} we discretize $I_i$ and use a continuity argument to conclude \eqref{correlation}. More precisely, we divide the interval $I_i$ into a mesh of $N^{10}$ equidistant points $\mathcal{P}_N$. For each $\theta,\widetilde{\theta}\in \mathcal{P}_N$ we have \eqref{correlation-two-theta}. If, instead, one of them $\theta\notin \mathcal{P}_N$, denoting by $\theta_*\in\mathcal{P}_N$ the closest element of $\mathcal{P}_N$ to $\theta$, we use (here ${\bf w}_i^\theta$ are the eigenvectors of $\mathcal{H}(\theta)$)
\[
\big|\lambda_1(\theta)-\lambda_1(\theta_*)\big|= \frac{1}{\sqrt{2}}\left|\int_{\theta}^{\theta_*} \langle {\bf w}_1^\tau, \big(e^{\ii\tau}A-e^{-\ii\tau}A^*\big){\bf w}_1^\tau\rangle\,\dd \tau\right|\le |\theta-\theta_*| \lVert A\rVert \le N^{-10},
\]
where we used that $\lVert A\rVert\lesssim 1$ with very high probability, and such an high probability event is independent of $\theta$.  Hence, \eqref{correlation-two-theta} holds for this $\theta$ as well. This gives \eqref{correlation}.

\end{proof}

We now present the proof of our claim \eqref{correlation-two-theta}.

\begin{proof}[Proof of \eqref{correlation-two-theta}]

We first prove \eqref{correlation-two-theta} in the Gaussian case, and then we extend it to the general case via comparison.

\textbf{Gaussian case.}  We present two different proofs. The first one relies on a strong control of quadratic forms of eigenvectors of Wigner matrices from \cite{CES21}. The second one is more elementary and self-contained. 

\smallskip

\textbf{Proof 1:} By first order perturbation theory
\[
\lambda_1(\theta)-\lambda_1(\theta_*)= \frac{1}{\sqrt{2}}\int_{\theta}^{\theta_*} \langle {\bf w}_1^\tau, \big(e^{\ii\tau}A-e^{-\ii\tau}A^*\big){\bf w}_1^\tau\rangle\,\dd \tau.
\]
Notice that when $A$ is Gaussian the matrices $\mathcal{H}(\tau)$ and $e^{\ii\tau}A-e^{-\ii\tau}A^*$ are independent. Hence, conditioning on $e^{\ii\tau}A-e^{-\ii\tau}A^*$, by \cite[Theorem 2.2]{CES21} (see also \cite[Theorem 2.2]{CEH23}) we have
\[
\big| \langle {\bf w}_1^\tau, \big(e^{\ii\tau}A-e^{-\ii\tau}A^*\big){\bf w}_1^\tau\rangle\big|\prec \frac{N^\xi}{\sqrt{N}},
\]
 for any arbitrary small $\xi>0$.  The uniformity in $\tau$ of the above estimate again simply follows from the continuity. This immediately shows \eqref{correlation-two-theta}.

\smallskip

\textbf{Proof 2:} Notice that we only need to prove two one-sided estimates
\begin{align}
\mathbb{P}\Big(\lambda_1(\theta)-\lambda_1(\tilde{\theta})\geq N^{-2/3-\tilde{\varepsilon}}\Big)\leq N^{-D}, \qquad \mathbb{P}\Big(\lambda_1(\tilde{\theta})-\lambda_1(\theta)\geq N^{-2/3-\tilde{\varepsilon}}\Big)\leq N^{-D}. \label{082605}
\end{align}
As we mentioned earlier, in the Gaussian case, $\mathcal{H}(\theta)$ is a stationary matrix-valued Gaussian process; see (\ref{091102}).   Since we are always at equilibrium, we can embed the two dimensional marginal of $\mathcal{H}(\theta)$  to a matrix valued (complex) OU process $X(t)$ (say) in either direction. More specifically,  as a consequence of
\[
\E \langle X(t), X(0)\rangle =e^{-t/2}, \qquad\quad \E \langle \mathcal{H}(\theta), \mathcal{H}(\tilde{\theta})\rangle=\cos(\theta-\tilde{\theta}),
\]
 we can view $(\mathcal{H}(\theta), \mathcal{H}(\tilde{\theta}))$ as $(X(0), X(t_0))$, or view $(\mathcal{H}(\tilde{\theta}), \mathcal{H}(\theta))$ as $(X(0), X(t_0))$, where $e^{-t_0/2}=\cos(\theta-\tilde{\theta})$, i.e. 
\begin{align*}
t_0\sim -\log \big(\cos(\theta-\tilde{\theta})\big)\sim (\theta-\tilde{\theta})^2. 
\end{align*}
In the sequel, we will take the first embedding, i.e, $((\mathcal{H}(\theta), \mathcal{H}(\tilde{\theta})))=(X(0), X(t_0))$ and use it to prove the first estimate in (\ref{082605}). The second estimate in (\ref{082605}) will follow from the same argument by using the other embedding. Hence, we can then embed $\lambda_1(\theta)$ and $\lambda_1(\tilde{\theta})$ to a (stationary) DBM so that $\theta$ corresponds to time 0 and $\tilde{\theta}$ corresponds to a $0<t_0\leq  N^{-1/3-2\varepsilon}$. We use $\mu_i(t)$ to denote the $i$-th largest eigenvalue in the DBM at time $t$. We recall the DBM for $\mu_1$:
\begin{align}
{\rm d}\mu_1= \frac{{\rm d} \beta_1}{\sqrt{N}}+\Big(\frac{1}{N}\sum_{\ell\geq 2} \frac{1}{\mu_1-\mu_\ell}-\frac{\mu_1}{2}\Big){\rm d} t, \label{082601}
\end{align}
 where $\beta_1$ is a standard real Brownian motion.  Fix any small $\tilde{\delta}>0$, we then write
 \begin{equation}
\begin{split}
\label{082602}
\frac{1}{N}\sum_{\ell\geq 2} \frac{1}{\mu_1-\mu_\ell}&=\frac{1}{N}\sum_{\ell=2}^{N^{\tilde{\delta}}} \frac{1}{\mu_1-\mu_\ell}+\frac{1}{N}\sum_{\ell>N^{\tilde{\delta}}} \frac{1}{\mu_1-\mu_\ell} \\
&>\frac{1}{N}\sum_{\ell>N^{\tilde{\delta}}} \frac{1}{\mu_1-\mu_\ell},
\end{split}
\end{equation}
where we used that $\mu_1-\mu_\ell>0$ for $\ell\ge 2$. 
 Next, using the rigidity estimate and local law in \cite[Theorems 2.1 and 2.2]{EYY12}, for $\eta:=N^{-2/3+\tilde{\delta}/2}$, we can write with very high probability
\begin{equation}
\label{eq:apprmu2}
\begin{split}
\frac{1}{N}\sum_{\ell>N^{\tilde{\delta}}} \frac{1}{\mu_1-\mu_\ell}=&\frac{1}{N}\sum_{\ell>N^{\tilde{\delta}}} \frac{1}{\mu_1+\ii\eta-\mu_\ell}+O\big(N^{-2/3+{\tilde{\delta}}}\big) \\
&=\frac{1}{N}\sum_{\ell=1}^N\frac{1}{\mu_1+\ii\eta-\mu_\ell}+O\big(N^{-1/3+{\tilde{\delta}}}\big) \\
&=-m_{\mathrm{sc}}(\mu_1+\ii\eta)+O\big(N^{-1/3+{\tilde{\delta}}}\big)\\
&=\frac{\mu_1}{2}+O\big(N^{-1/3+{\tilde{\delta}}}\big).
\end{split}
\end{equation}
Integrating \eqref{082601} over the time domain $[0, t_0]$ and using \eqref{082602}--\eqref{eq:apprmu2}, we thus obtain
\begin{equation}
\label{eq:lowbdiff}
\mu(t_0)-\mu(0)\ge \frac{\beta_1(t_0)}{\sqrt{N}}-t_0N^{-1/3+\tilde{\delta}}.
\end{equation}
Using standard large deviation estimates for the Brownian motion we obtain
\begin{equation}
\label{eq:boundBM}
|\beta_1(t_0)|\le N^\xi\sqrt{\frac{t_0}{N}}\sim N^{-2/3-\varepsilon+\xi},
\end{equation}
with very high probability. Recalling that $t_0\le N^{-1/3-2\varepsilon}$ and combining \eqref{eq:lowbdiff}--\eqref{eq:boundBM}, we conclude
\[
\mu(t_0)-\mu(0)\geq -CN^{-2/3-\tilde{\varepsilon}},
\]
for some $C>0$, and $\tilde{\varepsilon}:=(\varepsilon-\xi)\wedge(2\varepsilon-\tilde{\delta})$. Noticing that $\mu(t_0)-\mu(0)=\lambda_1(\tilde{\theta})-\lambda_1(\theta)$, this concludes the first relation in \eqref{082605}, and hence the proof.

\medskip

\textbf{General case.} We now prove \eqref{correlation-two-theta} for matrices with general entry distribution. This will be achieved via a comparison argument. We first relate the probability in \eqref{correlation-two-theta} to certain linear statistics of eigenvalues (see \eqref{eq:rellinstat} below) and then use a \emph{Green's function comparison argument (GFT)} for these statistics to show that they behave as in the Gaussian case which was already understood above. Notice that by the rigidity estimate  in \cite[Theorem 2.2]{EYY12},  we have
\begin{equation}
\begin{split}
\label{eq:prelim}
&\mathbb{P}\Big(\Big|\lambda_1(\theta)-\lambda_1(\tilde{\theta})\Big|\geq N^{-2/3-\tilde{\varepsilon}}\Big) \notag\\
&=\mathbb{P}\Big(\Big|\lambda_1(\theta)-\lambda_1(\tilde{\theta})\Big|\geq N^{-2/3-\tilde{\varepsilon}}, |\lambda_1(\theta)-2|\vee  |\lambda_1(\tilde{\theta})-2|\leq N^{-2/3+\tilde{\varepsilon}}\Big)+N^{-D}. 
\end{split}
\end{equation}
We then do the decomposition
\begin{align*}
&\big[2-N^{-2/3+\tilde{\varepsilon}}, 2+N^{-2/3+\tilde{\varepsilon}}\big]=\bigcup_{\ell=1}^{8N^{2\tilde{\varepsilon}}}J_{\ell}, \notag\\
&J_{\ell}=\big[2-N^{-2/3+\tilde{\varepsilon}}+\frac{1}{4}(\ell-1) N^{-2/3-\tilde{\varepsilon}}, 2-N^{-2/3+\tilde{\varepsilon}}+\frac{1}{4}\ell N^{-2/3-\tilde{\varepsilon}}\big]=:[r_{\ell-1}, r_\ell]. 
\end{align*}
We further notice that 
\begin{align*}
\mathcal{E}:=\Big\{ \big|\lambda_1(\theta)-\lambda_1(\tilde{\theta})\big|\geq N^{-2/3-\tilde{\varepsilon}}, |\lambda_1(\theta)-2|\vee  |\lambda_1(\tilde{\theta})-2|\leq N^{-2/3+\tilde{\varepsilon}}\Big\}= \bigcup_{\ell=1}^{8N^{2\tilde{\varepsilon}}}\mathcal{E}_{\ell},
\end{align*}
where 
\begin{align*}
\mathcal{E}_{\ell}:=\mathcal{E}\cap \big\{ \lambda_1(\theta)\in J_\ell\big\}. 
\end{align*}
Notice that
\begin{align}
\mathcal{E}_\ell\subseteq &\Big\{\lambda_1(\theta)\in [r_{\ell-1}, r_\ell]\Big\}\bigcap \Big\{\lambda_1(\tilde{\theta})\in [2-N^{-2/3+\tilde{\varepsilon}}, r_{\ell-3}]\cup [r_{\ell+2}, 2+N^{-2/3+\tilde{\varepsilon}}]\Big\}\notag\\
\subseteq &\Big\{\lambda_{1}(\theta)\in [r_{\ell-1}, 2+N^{-2/3+\tilde{\varepsilon}}], \lambda_{1}(\tilde{\theta})\leq r_{\ell-3}\Big\}\notag\\
&\bigcup \Big\{ \lambda_1(\tilde{\theta})\in [r_{\ell+2}, 2+N^{-2/3+\tilde{\varepsilon}}], \lambda_1(\theta)\leq r_{\ell}\Big\}, \label{082501}
\end{align}
where we make the convention that $\{\lambda_1(\tilde{\theta})\leq r_a\}=\emptyset$ if $a\leq 0$ and $\{\lambda_1(\tilde{\theta})\geq r_a\}=\emptyset$ if $a\geq 8N^{2\tilde{\varepsilon}}$. 
We set
\begin{align*}
L_{\ell}:= [r_{\ell}, 2+N^{-2/3+\tilde{\varepsilon}}], \qquad \ell=1,\ldots, 8N^{2\tilde{\varepsilon}}, 
\end{align*}
and we further make the convention that 
\begin{align*}
L_{\ell}:= [2-N^{-2/3+\tilde{\varepsilon}}, 2+N^{-2/3+\tilde{\varepsilon}}], \qquad \ell=-2,-1,0.
\end{align*} 
 We have
\begin{align}
\label{eq:rellinstat}
\text{RHS of (\ref{082501})}\subseteq &\Big\{\text{Tr}\chi_{L_{\ell-1}}(\mathcal{H}(\theta))\geq 1, \text{Tr} \chi_{L_{\ell-3}}(\mathcal{H}(\tilde{\theta}))=0\Big\} \notag\\
& \bigcup \Big\{ \text{Tr} \chi_{L_{\ell+2}}(\mathcal{H}(\tilde{\theta}))\geq 1, \text{Tr} \chi_{L_\ell}(\mathcal{H}(\theta))=0\Big\}.
\end{align}
Here we used the notation 
\begin{align*}
\text{Tr} \chi_{I}(\mathcal{H}(\theta))=\sum_{i=1}^{N} \mathbf{1}(\lambda_i(\theta)\in I). 
\end{align*}
In the Gaussian case,  by \eqref{correlation-two-theta} and the rigidity estimate in \cite[Theorem 2.2]{EYY12} we have  that the probability of both events in the right-hand side of \eqref{eq:rellinstat} is $N^{-D}$.  We now claim that
\begin{align}
&\mathbb{P}\left(\text{Tr} \chi_{L_{\ell+2}}(\mathcal{H}(\tilde{\theta}))\geq 1, \text{Tr} \chi_{L_\ell}(\mathcal{H}(\theta))=0\right)\notag\\
&\qquad+\mathbb{P}\left(\text{Tr}\chi_{L_{\ell-1}}(\mathcal{H}(\theta))\geq 1, \text{Tr} \chi_{L_{\ell-3}}(\mathcal{H}(\theta))=0\right)\le N^{-D}, \label{eq:claimgen}
\end{align}
for general matrices, for any $D>0$ and any $\ell\in\{1, 2, \dots, 8N^{2\tilde{\varepsilon}}\}$. We now just consider the first term in the left-hand-side of \eqref{eq:claimgen} since the estimate of the second one is completely analogous. This proof is similar to Section 3.2 of \cite{BCEHK25b}, we thus only explain the main steps for the convenience of the reader and omit the details.

 Let $F:\mathbb{R}_+\to\mathbb{R}_+$ be a smooth non-decreasing cut-off function such that $F(x)=0$ for $x\in [0,1/9]$ and $F(x)=1$ for $x\ge 2/9$. Then, it is easy to see that (here $\widetilde{F}:=1-F$)
\begin{align}
\mathbb{P}\left(\text{Tr} \chi_{L_{\ell+2}}(\mathcal{H}(\tilde{\theta}))\geq 1, \text{Tr} \chi_{L_\ell}(\mathcal{H}(\theta))=0\right)=\E\left[F\big(\text{Tr} \chi_{L_{\ell+2}}(\mathcal{H}(\tilde{\theta}))\big)\widetilde{F}\big(\text{Tr} \chi_{L_\ell}(\mathcal{H}(\theta))\big)\right]. \label{eq:expgft}
\end{align}
In order to apply a GFT argument we express the right-hand-side of \eqref{eq:expgft} in terms of resolvents. To this end, we set $G(z)\equiv G(\theta;z):=(\mathcal{H}(\theta)-z)^{-1}$ to be the resolvent. We define  
\begin{equation}
\label{eq:smooth}
\mathcal{X}_{\pm,\ell}(\mathcal{H}):= \chi_{[r_\ell\pm l, E_L]}*\xi_{\eta}\big(\mathcal{H}\big)=\frac{N}{\pi}\int_{r_\ell\pm l}^{E_L}\Im \mathrm{Tr} G(y+\ii\eta) \dd y.
\end{equation}
Here $E_L:=N^{-2/3+\tilde{\varepsilon}}$, $l:=N^{-2/3-2\tilde{\varepsilon}}$, $\eta:=N^{-2/3-20\tilde{\varepsilon}}$, and $\xi_\eta(x): =\eta/[\pi(x^2+\eta^2)]$ is a Cauchy mollifier.
Similarly to \cite[Lemma 2.3]{EX23}, we have
\begin{equation}
\label{eq:exptoresolvent}
	\begin{split}
&\E\left[F\big(\mathcal{X}_{+,\ell+2}(\mathcal{H}(\tilde{\theta}))\big)\widetilde{F}\big(\mathcal{X}_{-,\ell}(\mathcal{H}(\theta))\big)\right]+O(N^{-D}) \\
&\qquad\qquad\qquad\qquad\qquad\quad\le\E\left[F\big(\text{Tr} \chi_{L_{\ell+2}}(\mathcal{H}(\tilde{\theta}))\big)\widetilde{F}\big(\text{Tr} \chi_{L_\ell}(\mathcal{H}(\theta))\big)\right] \\
&\qquad\qquad\qquad\qquad\qquad\qquad\qquad\le\E\left[F\big(\mathcal{X}_{-,\ell+2}(\mathcal{H}(\tilde{\theta}))\big)\widetilde{F}\big(\mathcal{X}_{+,\ell}(\mathcal{H}(\theta))\big)\right]+O(N^{-D}),
	\end{split}
\end{equation}

Finally, we use an analogous GFT argument to \cite[Proposition 3.4]{BCEHK25b}. For this purpose, we consider the Ornstein-Uhlenbeck flow
\begin{equation}
\label{eq:OU}
\dd A_t=-\frac{1}{2}A_t\dd t+\frac{\dd B_t}{\sqrt{N}}, \qquad\quad A_0=A.
\end{equation}
Here $B_t$ is a matrix valued standard  Hermitian Brownian motion, i.e. $(B_t)_{ij}$ are i.i.d. standard complex Brownian motions for $i\ge j$ and $B_t^*=B_t$. Note that along the flow \eqref{eq:OU} the first two moments of $A_t$ are preserved. Hence, the limiting eigenvalue distribution of $\mathcal{H}_t(\theta)$ is given by the semicircular law for any $t\ge 0$, and $\mathcal{H}_\infty(\theta)\stackrel{\dd}{=}GUE$.

We are now ready to state our GFT argument. The proof of this result is presented at the end of this section. We now focus only on the term in the first line of \eqref{eq:exptoresolvent}, since the estimate of the one in the third line is completely analogous.
\begin{lem}
\label{lem:GFT}
Define
\begin{equation}
\label{eq:defrt}
R_t:=F\big(\mathcal{X}_{+,\ell+2}(\mathcal{H}_t(\tilde{\theta}))\big)\widetilde{F}\big(\mathcal{X}_{-,\ell}(\mathcal{H}_t(\theta))\big).
\end{equation}
Then, there exists $C>0$ such that for any $t\ge 0$ we have 
\begin{equation}
\label{eq:gronwall}
\big|\partial_t \E R_t\big|\lesssim N^{-1/6+C\tilde{\epsilon}}\E\left[F\big(\mathcal{X}_{-,\ell+2}(\mathcal{H}_t(\tilde{\theta}))\big)\widetilde{F}\big(\mathcal{X}_{+,\ell}(\mathcal{H}_t(\theta))\big)\right]+O(N^{-D}).
\end{equation}
\end{lem}
Note that the term in the rhs. of \eqref{eq:gronwall} differs from $R_t$, i.e. in \eqref{eq:smooth} $r_\ell+l$ is replaced by $r_\ell-l$, however this minor change can be neglected since $l$ is much smaller than $N^{-2/3-\tilde{\epsilon}}$ (see the explanation around \cite[Eq. (2.62)]{EX23} for more details). Integrating \eqref{eq:gronwall} in time we thus get
\begin{equation}
\label{eq:aftgron}
\big|\E\big[R_t-R_0\big]\big|\lesssim N^{-1/6+C\tilde{\epsilon}} \log N\big|\E R_t\big|+O(N^{-D}),
\end{equation}
for any $t\in [0,C_1\log N]$, for a large $C_1>0$. The regime $t\in (C_1\log N,\infty)$ can be trivially estimate by perturbation theory, since $\mathcal{H}_t\stackrel{\dd}{=}GUE+O(N^{-C_1/2})$, for $t\in (C_1\log N,\infty)$, where the error is meant in operator norm. Hence, \eqref{eq:aftgron} can be extended to the whole $t\in [0,\infty)$ regime. From \eqref{eq:aftgron}, we thus get
\begin{equation}
\label{eq:conclgronw}
\E\big[R_0\big]=\big(1+O(N^{-1/6+C\tilde{\epsilon}})\big)\E\big[R_t\big]+O(N^{-D}), \qquad\quad t \in[0,\infty).
\end{equation}
Finally, using \eqref{eq:conclgronw}, that the regime $t\in (C_1\log N,\infty)$ is negligible by a simple perturbation theory as above, and the fact that both sides of the inequality \eqref{eq:exptoresolvent} are bounded by $N^{-D}$, for any $D>0$, in the Gaussian case as a consequence of \eqref{correlation-two-theta}, \nc we readily obtain
\[
\left|\E\left[F\big(\mathcal{X}_{+,\ell+2}(\mathcal{H}(\tilde{\theta}))\big)\widetilde{F}\big(\mathcal{X}_{-,\ell}(\mathcal{H}(\theta))\big)\right]\right|+\left|\E\left[F\big(\mathcal{X}_{-,\ell+2}(\mathcal{H}(\tilde{\theta}))\big)\widetilde{F}\big(\mathcal{X}_{+,\ell}(\mathcal{H}(\theta))\big)\right]\right|\le N^{-D},
\]
for any $D>0$. This, together with \eqref{eq:prelim}--\eqref{eq:rellinstat} and \eqref{eq:exptoresolvent}, concludes the proof of \eqref{correlation-two-theta}.

\end{proof}

We conclude this section with the proof of Lemma \ref{lem:GFT}. 
\begin{proof}[Proof of Lemma~\ref{lem:GFT}]

The argument of this proof is similar to the proofs of \cite[Theorem 2.4]{EX23} and \cite[Proposition 3.4]{BCEHK25b}, \nc we will thus only explain the main steps. By It\^{o}'s formula we have
\begin{equation}
\partial_t \E R_t=\E\left[-\frac{1}{2}\sum_{a,b=1}^N \big[(A_t)_{ab}\partial_{ab} R_t+\overline{(A_t)_{ba}}\overline{\partial_{ba}} R_t]+\frac{1}{N}\sum_{a,b=1}^N \partial_{ab}\overline{\partial_{ab}} R_t\right],
\end{equation}
where $\partial_{ab}$ denotes the directional derivative in the direction $(A_t)_{ab}$. Performing a cumulant expansion (see e.g. \cite{HK17, KKP96}), for any $D>0$, choosing $K=K(D)$ sufficiently large, we get
\begin{equation}
\label{eq:aftercum}
\partial_t \E R_t=\sum_{k=2}^K\sum_{l=0}^{k+1} \sum_{a,b=1}^N\frac{\kappa_{ab}^{l,k}(t)}{k!}\partial_{ab}^l (\overline{\partial_{ab}})^{k+1-l}\E\big[ R_t\big]+O(N^{-D}),
\end{equation}
where $\kappa_{ab}^{l,k}(t)$ denotes the joint cumulant of $l$-copies of $(A_t)_{ab}$ and $k+1-l$-copies of $\overline{(A_t)_{ab}}$. In particular, we notice that only third or higher order cumulants remain in \eqref{eq:aftercum}.

Notice, that \eqref{eq:aftercum} corresponds to \cite[Eq. (3.16)]{BCEHK25b}. The only input used in \cite{BCEHK25b} to estimate the rhs. of  \eqref{eq:aftercum} is the single resolvent local law
\[
\max_{i,j\in [N]}\big|\big(G^\theta(z)-m_{\mathrm{sc}}(z)\big)_{ij}\rangle\big|\prec N^{-1/3+C\tilde{\epsilon}}, \qquad\quad |\Im m_{\mathrm{sc}}(z)|\sim N^{-1/3+C\tilde{\epsilon}},
\]
which holds in our case as well \cite[Theorem 2.1]{EYY12}.  Here $C>0$ is a large constant which may change from line to line. In particular, no relation of the spectra of $\mathcal{H}(\theta), \mathcal{H}(\tilde{\theta})$ is needed to estimate the rhs. of \eqref{eq:aftercum}, i.e. it is enough to only control the single resolvent of these matrices, no products of resolvents appear. Hence, proceeding as in \cite[Eqs. (3.18)--(3.23)]{BCEHK25b} we obtain
\begin{equation}
\label{eq:almfinGFT}
\partial_t \E R_t\lesssim N^{-1/6+C\tilde{\epsilon}}\sum_{1\le p+q\le K+1}\E\left[F^{(p)}\big(\mathcal{X}_{-,\ell+2}(\mathcal{H}_t(\tilde{\theta}))\big)\widetilde{F}^{(q)}\big(\mathcal{X}_{+,\ell}(\mathcal{H}_t(\theta))\right]+N^{-D}.
\end{equation}
Here for a function $f(x)$ by $f^{(p)}(x)$ we denote its $p$-th derivative. Then, by construction of $F,G$ it follows (see \cite[Eq. (3.24)]{BCEHK25b})
\begin{equation}
\label{eq:finGFT}
\E\left[F^{(p)}\big(\mathcal{X}_{-,\ell+2}(\mathcal{H}_t(\tilde{\theta}))\big)G^{(q)}\big(\mathcal{X}_{+,\ell}(\mathcal{H}_t(\theta))\right]\lesssim\E\left[F\big(\mathcal{X}_{+,\ell+2}(\mathcal{H}_t(\tilde{\theta}))\big)G\big(\mathcal{X}_{-,\ell}(\mathcal{H}_t(\theta))\big)\right]+N^{-D},
\end{equation}
where the implicit constant in the error term depends only on $p$ and $q$. Finally, combining \eqref{eq:almfinGFT}--\eqref{eq:finGFT} we conclude the proof.
\end{proof}

\section{Decorrelation Regime: Proof of Proposition~\ref{lem.decorrelation}} \label{s.decorrelation}

A natural way to prove the decorrelation estimate in Proposition~\ref{lem.decorrelation} is to proceed similarly to \cite[Sections 4-5]{BCEHK25a}, which relies on the Zigzag strategy from \cite{CEH23}. However, we now show that in the current case the estimate in Proposition~\ref{lem.decorrelation} can be achieved in a simpler way (see \cite{K25} for a related argument in the bulk of the spectrum). Let $\theta_1,\theta_2\in [0,\pi/2]$. We first consider the Gaussian case, and then we use a standard Green's function comparison theorem to extend this result to the general case. When $A$ is a complex Ginibre matrix, by \eqref{091102}, it follows that:
\begin{equation}
\begin{split}
\mathcal{H}(\theta_1)&\stackrel{\dd}{=}\sqrt{\cos(\theta_1-\theta_2)}\mathcal{G}_1+\sqrt{1-\cos(\theta_1-\theta_2)}\mathcal{G}_2, \\
\mathcal{H}(\theta_2)&\stackrel{\dd}{=}\sqrt{\cos(\theta_1-\theta_2)}\mathcal{G}_1+\sqrt{1-\cos(\theta_1-\theta_2)}\mathcal{G}_3,
\end{split} \label{092301}
\end{equation}
where $\mathcal{G}_1,\mathcal{G}_2,\mathcal{G}_3$ are three independent GUE matrices.

\begin{proof}[Proof of Proposition~\ref{lem.decorrelation}]  

Consider $\theta_1,\theta_2$ such that $|\theta_1-\theta_2|\ge N^{-1/6+\epsilon}$. We only present the proof of the decorrelation estimate for $\mathcal{F}_{\theta_1}(x_1), \mathcal{F}_{\theta_2}(x_2)$; the proof of the decorrelation estimate for $\mathcal{E}_{\theta_1}(x_1), \mathcal{E}_{\theta_2}(x_2)$ is completely analogous and so omitted.

For $l=1,2$, consider the flows:
\begin{equation}
\label{eq:matriflow}
\dd \mathcal{H}_t(\theta_l)=\frac{\dd B_{l,t}}{\sqrt{N}}, \qquad\quad \mathcal{H}_0(\theta_l)=\sqrt{\cos(\theta_1-\theta_2)}\mathcal{G}_1,
\end{equation}
where $B_{l,t}$ are two independent matrix valued Hermitian Brownian motions. Note that the initial condition in \eqref{eq:matriflow} is the same for both flows. We will run these flows up to a time $t_1= N^{-1/3+\omega_1}$ so that $\mathcal{H}_{t_1}(\theta_l)=\mathcal{H}(\theta_l)$, for $l=1,2$, i.e. we see the matrices $\mathcal{H}(\theta_l)$ as the final time of the matrix flow \eqref{eq:matriflow}.  The eigenvalues $\mu_i^{\theta_l}(t)$ of $\mathcal{H}_t(\theta_l)$ are the unique strong solution of the Dyson Brownian motion (see \cite{AGZ10, D62}): 
\begin{equation}
\label{eq:newDBM}
\dd \mu_i^{\theta_l}(t)=\frac{\dd b_i^{\theta_l}(t)}{\sqrt{N}}+\frac{1}{N}\sum_{j\ne i}\frac{1}{\mu_i^{\theta_l}(t)-\mu_j^{\theta_l}(t)}\dd t,
\end{equation}
where $b_i^{\theta_l}(t)$, for $i\in [N]$ and $l=1,2$, are an i.i.d. family of standard real Brownian motions. Since the evolution of the $\mu_i^{\theta_l}(t)$ in \eqref{eq:newDBM} is independent for $\theta_1,\theta_2$, the flows in \eqref{eq:newDBM} can be directly coupled to two fully independent processes. More precisely, we consider 
\begin{equation}
\label{eq:newDBM2}
\dd \nu_i^{(l)}(t)=\frac{\dd b_i^{\theta_l}(t)}{\sqrt{N}}+\frac{1}{N}\sum_{j\ne i}\frac{1}{\nu_i^{(l)}(t)-\nu_j^{(l)}(t)}\dd t,
\end{equation}
with independent initial data $\nu_i^{(l)}(0)$. Note that the driving Brownian motions are exactly the same as in \eqref{eq:newDBM}. Fix a small $c>0$. Then, by \cite[Theorem 2.8]{B21} we have
\begin{equation}
\sup_{1\le i \le N^c}\big|\mu_i^{\theta_l}(t_1)-\nu_i^{(l)}(t_1)\big|\le N^{-2/3-\omega_1/10}
\end{equation}
with very high probability, where $t_1:=N^{-1/3+\omega_1}$, for a fixed small $\omega_1>0$ which does not depend on $\varepsilon$ if $|\theta_1-\theta_2|\geq N^{-1/12}$ (say). Similarly to \cite[Lemma 3.4]{BCEHK25b} and the discussion around it, the error $N^{-2/3-\omega_1/10}$ can  be further ignored in the decorrelation relations in Proposition~\ref{lem.decorrelation}, up to a negligible error. This shows the desired decorrelation estimate in Proposition~\ref{lem.decorrelation} in the Gaussian case.

To conclude the proof, we perform a long time (i.e. $t\sim \log N$) GFT to transfer the decorrelation estimate Proposition~\ref{lem.decorrelation} from the Gaussian case to the general case. We notice that
\begin{equation}
\mathbb{P}\Big(\mathcal{F}_{\theta}(x)\cap \mathcal{F}_{\theta'}(x')\Big)=\E\left[F\big(\text{Tr} \chi_{[E_1,2+N^{-2/3+\varepsilon}]}(\mathcal{H}(\tilde{\theta}))\big)F\big(\text{Tr} \chi_{[E_2,2+N^{-2/3+\varepsilon}]}(\mathcal{H}(\theta))\big)\right]+O(N^{-D}),
\end{equation}
where for $l=1,2$ we denoted $E_l:=2+x_l(\log N)^{2/3}/N^{2/3}$. Hence, by bounds similar to \eqref{eq:exptoresolvent}, the proof of this GFT is completely analogous to the proof of Lemma~\ref{lem:GFT} above (see also the discussion in \cite[Eqs. (3.21)--(3.28)]{BCEHK25b}) with the only difference that  $\widetilde{F}$ needs to be replaced with $F$ (from above \eqref{eq:expgft}) in the definition of $R_t$ in \eqref{eq:defrt}, and thus omitted.
\end{proof}

\section{Elliptic ensemble}\label{s.elliptic}

In this section, we consider the elliptic ensemble in (\ref{090801}) and prove Theorem \ref{thm.elliptic}. 

\begin{proof}[Proof of Theorem \ref{thm.elliptic}]
We set 
\begin{align*}
H^{\gamma}(\theta)=(h_{ij}^{\gamma}(\theta)):= \big(e^{\ii \theta} A_{\gamma}+e^{-\ii \theta} A_{\gamma}^*\big)/2. 
\end{align*}
and we further denote by $\lambda^{\gamma}_1(\theta)$ the largest eigenvalue of $H^{\gamma}(\theta)$. 
It is elementary to check 
\begin{align*}
\mathbb{E}|h_{ij}^{\gamma}(\theta)|^2=\frac{1+{\gamma}\cos 2\theta}{2N}, \qquad \mathbb{E}h_{ij}^2(\theta)=0
\end{align*}
and 
\begin{align*}
\mathbb{E} h_{ij}^{\gamma}(\theta)\overline{h_{ij}^{\gamma}(\theta')}=\frac{\cos(\theta-\theta')+{\gamma}\cos (\theta+\theta')}{2N}, \qquad \mathbb{E} h_{ij}^{\gamma}(\theta){h_{ij}^{\gamma}(\theta')}=0.
\end{align*}
Especially, when $W$ is GUE and $A$ is complex Ginibre, $H^{\gamma}(\theta)$ is still a (scaled) GUE, but the variance will depend on $\theta$. When $\theta=0 \text{ or } \pi$, the variance reaches the maximum $(1+{\gamma})/2N$, while when $\theta=\pi/2 \text { or } 3\pi/2$, the variance reaches the minimum $(1-{\gamma})/2N$. Due to the change of the matrix entry variance, $\lambda_1^{\gamma}(\theta)$ is no longer stationary, as a process of $\theta$. The nonstationarity allows us to localize our discussion to small vicinities around $0$ and $\pi$ for $r_+(A^{\gamma})$, and those around $\pi/2$ and $3\pi/2$ for $r_-(A^{\gamma})$.  We further recall the assumptions 
\begin{align}
{\gamma}\in [N^{-1/3+\delta},1]\quad \text{or}\quad  {\gamma}\in [N^{-1/3+\delta},1-\delta']
\end{align}
in Theorem \ref{thm.elliptic} (i) or (ii), respectively.

By the rigidity estimate in \cite[Theorem 2.2]{EYY12}, we know
\begin{align*}
\mathbb{P}\Big(\big|\lambda_1^{\gamma}(\theta)- \sqrt{2(1+{\gamma}\cos 2\theta)}\big|\geq N^{-2/3+\delta/2} \Big)\leq N^{-D}
\end{align*}
Hence, by a simple continuity argument,  we have with very high probability
\begin{align*}
\max_{\theta\in [0,2\pi)}\lambda_1^{\gamma}(\theta)= \max_{\theta\in I_{0, \pi}}\lambda_1^{\gamma}(\theta), \qquad \min_{\theta\in [0,2\pi)}\lambda_1^{\gamma}(\theta)= \min_{\theta\in I_{\frac{\pi}{2}, \frac{3\pi}{2}}}\lambda_1^{\gamma}(\theta)
\end{align*}
where
\begin{align*}
I_{a, b}:=[a-N^{-1/6-\varepsilon}, a+N^{-1/6-\varepsilon}]\cup [b-N^{-1/6-\varepsilon}, b+N^{-1/6-\varepsilon}].
\end{align*}
Here we chose $0<\varepsilon=\varepsilon(\delta)<\delta/4$. 
We can then do a rescaling and use the correlation estimate similarly to (\ref{correlation-two-theta}) to conclude
\begin{align*}
\mathbb{P}\Big(\Big|\max_{\theta\in I_{0, \pi}}\sqrt{\frac{1+{\gamma}}{1+{\gamma}\cos 2\theta}}\lambda_1^{\gamma}(\theta)-\max_{\theta=0, \pi} \lambda_1^{\gamma}(\theta)\Big|\geq N^{-2/3-\tilde{\varepsilon}}\Big)\leq N^{-D}
\end{align*}
Consequently, we have 
\begin{align*}
\mathbb{P}\Big(\Big|\max_{\theta\in I_{0, \pi}}\lambda_1^{\gamma}(\theta)-\max_{\theta=0, \pi} \lambda_1^{\gamma}(\theta)\Big|\geq N^{-2/3-\tilde{\varepsilon}}\Big)\leq N^{-D}
\end{align*}
Similarly, we also have 
\begin{align*}
\mathbb{P}\Big(\Big|\min_{\theta\in I_{\frac{\pi}{2}, \frac{3\pi}{2}}}\lambda_1^{\gamma}(\theta)-\min_{\theta=\frac{\pi}{2}, \frac{3\pi}{2}} \lambda_1^{\gamma}(\theta)\Big|\geq N^{-2/3-\tilde{\varepsilon}}\Big)\leq N^{-D}
\end{align*}
From \cite{EYY12}, we know that $\lambda_1^{\gamma}(\theta)$ follows Tracy-Widom law with suitable parameter. It is thus enough to show the asymptotic independence of the fluctuations of $\lambda_1^{\gamma}(0)$ and $\lambda_1^{\gamma}(\pi)$, as well as the fluctuations  of $\lambda_1^{\gamma}(\frac{\pi}{2})$ and $\lambda_1^{\gamma}(\frac{3\pi}{2})$. Notice that $H^{\gamma}(0)=-H^{\gamma}(\pi)$. Hence, $\lambda_1^{\gamma}(0)$ and $-\lambda_1^{\gamma}(\pi)$ are the largest and smallest eigenvalues of $H^{\gamma}(0)$, respectively. The relation between $\lambda_1^{\gamma}(\frac{\pi}{2})$ and $\lambda_1^{\gamma}(\frac{3\pi}{2})$ is analogous. Hence, it boils down to show the asymptotic independence of the fluctuations of the  largest and smallest eigenvalues of a Wigner matrix. This can be obtained via a simple GFT with the results in \cite{Bor10} or \cite{Sa99}.  We omit the details. This concludes the proof of Theorem \ref{thm.elliptic}.

\end{proof}

\section{Further discussion} \label{s.discussion}

In this section, we comment on some future directions. 

\subsection{Further expansion of numerical radii} In Theorem \ref{maintheorem}, we obtained the first and second order terms of $r_{\pm}(A)$. A natural subsequent question concerns the next order terms, up to the fluctuation term. For the first two orders, our analysis relied on a correlation-decorrelation transition at the scale $N^{-1/6\pm \varepsilon}$. There remains an $N^{2\varepsilon}$ scale gap between the correlation and the decorrelation regimes. To obtain the next order term using the method in this work, this gap shall first be reduced to a much finer one, for instance, of logarithmic order. A correlation-decorrelation transition on a finer scale is not only necessary for a better understanding of $r_{\pm}(A)$ but is also an intrinsically interesting question in its own right. We will consider it in a future work. 

\subsection{Real non-Hermitian random matrix} In this work, we consider complex non-Hermitian random matrices. It is natural to ask whether similar results can be established for real matrices as well. Our analysis for the complex case relies on the Gaussian setting via the GFT method, under which the matrix $\mathcal{H}(\theta)$ is always a GUE matrix for any $\theta$. When we consider the real Ginibre matrix, the matrix $\mathcal{H}(\theta)$ undergoes a transition from GOE to GUE and further to antisymmetric GUE as $\theta$ increases from $0$ to $\frac{\pi}{2}$ and then to $\pi$, and then goes through the transition backwards when $\theta$ goes from $\pi$ to $2\pi$. The tail probability estimate in Proposition \ref{pro.tail} is not directly available for the interpolation between GOE, GUE, and antisymmetric GUE. However, we note that for $\theta\in [N^{-1/6+\varepsilon}, \pi-N^{-1/6+\varepsilon}]\cup [\pi+N^{-1/6+\varepsilon}, 2\pi-N^{-1/6+\varepsilon}]$, one can view the interpolations as a GOE or antisymmetric GUE perturbed by a certain amount of GUE. Then, by a DBM analysis, one can expect that in this regime, the tail probability estimate for the largest eigenvalue of $\mathcal{H}(\theta)$ should resemble that of the GUE. Nevertheless, such an argument still leaves the complementary regime unresolved. Apart from this issue, all other arguments in this work can be carried over to the real case straightforwardly.

\subsection{General matrix-valued process} In this work, our analysis essentially reduces to studying the extrema of the largest eigenvalue of a matrix-valued process defined in (\ref{091102}), which is a Gaussian process when $A$ is a complex Ginibre matrix. Another widely studied Gaussian process is the matrix-valued Brownian motion, whose eigenvalue process is the DBM, as we have mentioned and used in our derivations. Our argument can actually be extended to study the extrema of the largest eigenvalue of a general matrix-valued Gaussian process, and further, to processes involving generally distributed random matrices.
For instance, a natural model to consider is the matrix-valued random Fourier series, whose scalar version has been extensively studied. Specifically, for instance one can consider the following model:
\begin{align}
F(\theta)= \sum_{k=-\infty}^{\infty} a_k A_k e^{\ii k \theta}, \qquad a_kA_k=\overline{a_{-k}A_{-k}}
\end{align}
where $a_{k}$'s are deterministic coefficients that decay sufficiently fast and $A_{k}$'s are independent non-Hermitian random matrices. In our setting, we consider the extrema of the largest eigenvalue of such a process over a large time domain. This differs from results on the extrema of DBM on the time scale $t\sim N^{-1/3}$, which are well understood from the KPZ literature \cite{Joh03, CH14}. Note that on the time scale $t\sim N^{-1/3}$, where decorrelation has not yet occurred, the analysis is intrinsically different from the case of random variable sequence with short-range correlation.

\subsection{Transition regime for Elliptic model} In this work, we consider two matrix models: the standard complex non-Hermitian random matrix, for which $\lambda_1(\theta)$ is stationary, and the elliptic model with $\gamma\gg N^{-1/3}$, where the process $\lambda_1(\theta)$ is sufficiently non-stationary. In the latter case, the result resembles the numerical radius of a Hermitian matrix, which by definition is simply the operator norm. The stationary and non-stationary cases discussed in this paper can thus be regarded as two typical scenarios for the problem of the numerical radii of random matrices. It is nevertheless interesting to consider the transition regime $\gamma\sim N^{-1/3}$ for the elliptic model. In this regime, the decay of the first-order term of $\lambda_1(\theta)$ allows us to localize the discussion to a vicinity of size $N^{-1/6}$ around $0$ and $\pi$ for $r_{+}(A^{\gamma})$, but not to a smaller one. For such a vicinity, we can no longer use the correlation estimate to reduce the discussion to only the two points $\theta=0$ and $\pi$. The non-trivial correlation on the scale $N^{-1/6}$ will then play a role. This situation is very similar to that in \cite{CH14, Joh03}. From a mathematical perspective, it would be interesting to investigate the distribution of $r_{\pm}(A)$ in this case.

\section*{Acknowledgement}

Z. Bao would like to thank Kevin Schnelli, Dong Wang, and Lun Zhang for helpful discussions, and Shuyang Ling and Qiang Zeng for helpful communications.

\end{document}